\documentclass[12pt,a4paper]{article}

\usepackage[T1]{fontenc}       
\usepackage[utf8]{inputenc}

\usepackage{microtype}

\usepackage[margin=2.5cm]{geometry}

\usepackage{mathtools}    
\usepackage{amssymb}           
\usepackage{amsthm}            
\usepackage{bm}             
\usepackage{mathrsfs}         
\usepackage{dsfont}

\usepackage{graphicx}          
\usepackage{xcolor}            
\usepackage{float}

\usepackage{tikz}
\usetikzlibrary{arrows,calc,positioning,external}

\usepackage[colorlinks, linkcolor=black, citecolor=black, urlcolor=black]{hyperref}

\usepackage[nameinlink,noabbrev]{cleveref}

\usepackage{caption}
\usepackage{subcaption}
\usepackage{wrapfig}

\usepackage[shortlabels]{enumitem}

\usepackage{authblk}

\usepackage[backend=biber, style=alphabetic, maxbibnames=40, doi=false, isbn=false, url=false, eprint=false]{biblatex}
\numberwithin{equation}{section}

\theoremstyle{plain}
\newtheorem{theorem}{Theorem}[section]

\newtheorem{lemma}[theorem]{Lemma}
\newtheorem{corollary}[theorem]{Corollary}

\theoremstyle{definition}

\newtheorem{assumption}[theorem]{Assumption}

\newtheorem{remark}[theorem]{Remark}
\newtheorem{example}[theorem]{Example}

\newcommand{\numberset}{\mathbb}
\newcommand{\N}{\numberset{N}}

\newcommand{\R}{\numberset{R}}

\DeclarePairedDelimiter{\abs}{\lvert}{\rvert}

\DeclarePairedDelimiter{\ceil}{\lceil}{\rceil}
\DeclarePairedDelimiter{\tond}{(}{)} 
\DeclarePairedDelimiter{\quadr}{[}{]}
\DeclarePairedDelimiter{\graf}{\{}{\}} 
\DeclarePairedDelimiter{\scal}{\langle}{\rangle}

\DeclareMathOperator{\dist}{d}

\DeclareMathOperator{\law}{law}

\DeclareMathOperator{\poi}{PI}

\DeclareMathOperator{\tph}{T_pH}
\DeclareMathOperator{\deftth}{{T_{2}}H}

\DeclareMathOperator{\deftph}{{T_{p}}H}

\DeclareMathOperator{\wmix}{w_{mix}}

\newcommand{\tmix}{\mathrm{t}_{\mathrm{mix}}}

\newcommand{\eps}{\varepsilon}

\newcommand{\tmixa}[1]{\tmix\tond*{#1}}

\newcommand{\wmixa}[1]{\wmix\tond*{#1}}

\newcommand{\dd}{\,d} 

\newcommand{\hess}{\nabla^2}

\newcommand{\expo}[1]{\exp\tond*{#1}}

\newcommand{\variwr}[2]{\mathrm{Var}_{#1}\tond*{#2}}

\newcommand{\expe}[1]{\mathbb{E}\quadr*{#1}}

\newcommand{\condexpe}[2]{\mathbb{E}\quadr*{#1 \,\middle|\, #2}}

\newcommand{\kldiv}[2]{\mathcal{H}\tond*{#1\,\middle| \,  #2}}

\newcommand{\pp}{\mathcal{P}}

\newcommand{\relden}[2]{\frac{d#1}{d#2}}

\newcommand{\cpi}[1]{C_{\poi}\tond*{#1}}

\newcommand{\pib}{\boldsymbol{\Pi}}

\newcommand{\defttha}[1]{\deftth\tond*{#1}}

\newcommand{\deftpha}[1]{\deftph\tond*{#1}}
\newcommand{\tpha}[1]{\tph\tond*{#1}}

\newcommand{\cdtth}[1]{C_\mathrm{\tilde{T}_2H}}

\newcommand{\ldsem}{P^{\mathrm{LD}}}
\newcommand{\lmcsem}[1]{P^{\mathrm{LMC_{#1}}}}
\newcommand{\gdsem}[1]{P^{\mathrm{GD_{#1}}}}

\newcommand{\pssem}[1]{P^{\mathrm{PS}_{#1}}}
\newcommand{\gausem}[1]{P^{\mathrm{BM}_{#1}}}
\newcommand{\revgausem}[1]{P^{\mathrm{Rev}_{#1}}}

\newcommand{\mubf}{\pmb{\mu}}
\newcommand{\nubf}{\pmb{\nu}}

\newcommand{\lesiep}{\lesssim_\eps}

\newcommand{\totvar}[2]{\mathrm{TV}\tond*{#1,  #2}}

\newcommand{\hatcpi}[1]{\widehat{C}_{\rm{P}}\tond*{#1}}

\title{Entropy-Wasserstein regularization, defective local concentration and a cutoff criterion beyond non-negative curvature}
\author{Francesco Pedrotti\thanks{\textsf{francesco.pedrotti@stat.math.ethz.ch}} }
\affil{ETH Zürich}

\begin{document}
	
	\maketitle
	
	\begin{abstract}
        Notions of positive curvature have been shown to imply many remarkable properties for Markov processes, in terms, e.g., of regularization effects, functional inequalities, mixing time bounds, and, more recently, the cutoff phenomenon.
        In this work, we are interested in a relaxed variant of Ollivier's coarse Ricci curvature, where a Markov kernel $P$ satisfies only a weaker Wasserstein bound $W_p(\mu P, \nu P) \leq K W_p(\mu,\nu)+M$ for constants $M\ge 0, K\in [0,1], p \ge 1$. 
        Under appropriate additional assumptions on the one-step transition measures $\delta_x P$, we establish (i) a form of local concentration, given by a defective Talagrand inequality, and (ii) an entropy-transport regularization effect.
        We consider as illustrative examples the Langevin dynamics and the Proximal Sampler when the target measure is a log-Lipschitz perturbation of a log-concave measure.
        As an application of the above results, we derive criteria for the occurrence of the cutoff phenomenon in some negatively curved settings.
	\end{abstract}
	
	\begingroup
    \enlargethispage{3\baselineskip}
    \tableofcontents
    \endgroup

	\section{Introduction}

    Let $(\Omega,\dist)$ be a Polish space equipped with its Borel $\sigma$-algebra and consider a Markov process $(X_k)_k$ (either in discrete or continuous time) taking values in $\Omega$.
    Understanding the evolution and ergodicity properties of this process, in a broad sense, is a fundamental problem, which has motivated the development of a variety of different techniques, often drawing inspiration from different fields and thus creating fascinating connections.

    \subsection{Curvature of Markov chains}
    One notable example of this is given by the geometric idea of \emph{curvature}. The seminal observations dating back to Lévy \cite{lev-1951} and Gromov \cite{mil-sch-gro-1986}, about the remarkable properties satisfied by the uniform measure on a positively curved Riemannian manifold, have in fact inspired the development of more general notions of curvature, that apply, e.g., to (some) Markov  processes or to abstract metric measure spaces. We recall in particular the celebrated theories by Bakry--Émery \cite{bak-eme-1985}, by Ollivier \cite{oll-2009}, by Lott--Villani \cite{lot-vil-2009} and Sturm \cite{stu-2006}, and by Erbar--Maas \cite{erb-maa-2012} and Mielke \cite{mie-2013}.
    While different in many ways and typically not comparable, a common feature of these theories is that positive curvature is linked to favorable properties of the reference measure and/or stochastic process, in terms e.g. of concentration of measure, functional inequalities and ergodicity.
    In this work, we are mostly concerned with the \emph{coarse Ricci curvature} by Ollivier  \cite{oll-2009} (and natural variants of it), which has become very popular because of its simple formulation. To define it, recall for $p \ge 1$ the $p$-Wasserstein distance between probability measures $\mu,\nu \in \pp_p\tond*{\Omega}$ with finite $p$-moment:
    \[
        W_p(\mu,\nu) = \inf_{\gamma \in \Gamma(\mu,\nu)} \graf*{\int \dist(x,y)^p \dd \gamma(x,y)}^{\frac1p},
    \]
    where $\Gamma(\mu,\nu)$ is the set of couplings between $\mu,\nu$, i.e. probability measures $\gamma\in \tond*{\Omega \times \Omega}$ whose  marginals are  $\mu$ and  $\nu$.
    Then, a discrete-time Markov chain has coarse Ricci curvature (of order $p$) greater than a constant $1-K$ if the Markov kernel $P$ contracts the $p$-Wasserstein distance, in the sense that for all $x,y \in \Omega$
    \begin{equation}\label{eq:p-coarse-ricci-geq-K}
        W_p(\delta_x P, \delta_y P) \le K \dist(x,y).
    \end{equation}
    This definition is naturally extended to continuous time semigroups $P_t$ by requiring the analogous condition
    \begin{equation}\label{eq:p-coarse-ricci-cts-geq-K}
        W_p(\delta_x P_t, \delta_y P_t) \le e^{-(1-K)t} \dist(x,y)
    \end{equation}
    for any $t\ge 0$.
    \subsection{Examples}
    To make the discussion more concrete, we focus now on some Markov processes in $\R^d$, whose limiting distribution is a probability measure $\pi \propto e^{-U}$ for a potential $U\in C^2(\R^d)$ (as usual, we identify probability measures with their density when they are absolutely continuous with respect to the Lebesgue measure).

    \paragraph{Langevin dynamics}
    The first canonical example that we consider is the \emph{Langevin dynamics}: given a standard $d$-dimensional Brownian motion $(B_t)_t$, it is the solution of the SDE
    \begin{equation} \label{eq:LD}
        dX_t = -\nabla U(X_t) \dd t + \sqrt{2} dB_t,
    \end{equation}
     whose semigroup we denote by $\ldsem_t$. For this process, curvature is very well understood (in most senses, e.g. Ollivier, Bakry--\'Emery, Lott--Sturm--Villani), and admits a remarkably simple characterization: the Ricci curvature of \eqref{eq:LD} is bounded from below by a constant $\alpha>0$ if and only if $\pi$ is $\alpha$-log concave, i.e. $\hess U \succcurlyeq \alpha I_d$. When this happens, as anticipated, several interesting consequences can be derived.
    For example, recall the definition of relative entropy
    \begin{equation}
\kldiv{\mu}{\pi} = \begin{cases}
    \int \log\tond*{\relden{\mu}{\pi}} \mathrm{d}\mu &\text{ if } \mu \ll \pi,
    \\
    + \infty &\text{ otherwise}.
\end{cases} 
	\end{equation}
    Then, if $\pi$ is $\alpha$-log-concave, the following celebrated regularization estimate holds for $\mu,\nu \in \pp_2 \tond*{\R^d}$ and $t>0$ \cite{bob-gen-led-2001}:
    \begin{equation}\label{eq:KL-W-log-conc-reg}
           \kldiv{\mu \ldsem_t}{\nu \ldsem_t} \leq \frac{\alpha e^{-2\alpha t}}{2(1-e^{-2\alpha t})}W_2^2(\mu,\nu).
    \end{equation}

    A second fundamental consequence of log-concavity is the validity of functional inequalities for the local law of \eqref{eq:LD} at any finite time $t$ rather than just for the limiting measure $\pi$. For example, for any $t\ge 0$ and bounded $f\in C^1(\R^d)$ we have the inequality
    \begin{equation}\label{eq:local-poincare-curved}
        P_t f^2- \tond*{P_t f}^2 \leq \frac{1-e^{-2\alpha t}}{\alpha}P_t \abs*{\nabla f}^2,
    \end{equation}
    which classically expresses the fact that the local law $\delta \ldsem_t$ satisfies a Poincaré inequality \cite{bak-gen-led-2014}.

    In this paper our focus is mostly on discrete-time Markov chains, so it is also useful to consider  the Euler--Maruyama discretization of \eqref{eq:LD}, for a small step-size $h>0$, which is given by recursion
	\begin{equation}\label{eq:LMC}
		X_{k+1} = X_k - h\nabla U(X_k) + \sqrt{2}\tond*{B_{(k+1)h}-B_{kh}},
	\end{equation}
	and is typically referred to as \emph{Langevin Monte Carlo algorithm}.
	We write $\gausem{h}$ for the transition kernel corresponding to a Gaussian convolution with covariance $hI_d$, i.e.
	$
    \delta_x \gausem{h} = \mathcal{N}(x,h I_d),
	$
	and $\gdsem{h}$ for the deterministic update corresponding to a gradient descent step for $U$, i.e.
	$
			\delta_x \gdsem{h} = \delta_{x-h\nabla U(x)}.
	$
	With this notation, the one-step transition kernel associated to \eqref{eq:LMC} is given by the composition
	$
			\lmcsem{h} \coloneqq \gdsem{h}\gausem{2h}.
	$
    
    \paragraph{Proximal sampler}
	The second example that we consider is the \emph{Proximal Sampler}, a popular
    discrete-time algorithm for sampling from $\pi$  introduced in \cite{lee-she-tia-2021}. 	We  refer the reader to \cite[Chapter 8]{che-book-2024+} or the recent papers \cite{che-eld-2022, mit-wib-2025, wib-2025} for more details. 
	Given a step size $h>0$, we  consider the following probability $\pib \in \pp\tond*{\R^d \times \R^d}$, whose first marginal is $\pi$:
	\begin{eqnarray}
		\pib(x,y) & \propto &  \exp\tond*{-U(x) - \frac{\abs*{x-y}^2}{2h}}.
	\end{eqnarray}
	The Proximal Sampler with target $\pi$ consists in applying alternating Gibbs sampling to $\pib$.
	More precisely, the Proximal Sampler transition kernel is defined by the composition
	 \[
	\pssem{h} \coloneqq \gausem{h}\revgausem{h},
	\]
	where the first \emph{forward step} corresponds to the Gaussian convolution described before, and the second \emph{backward step} $\revgausem{h}$ is defined by
	\[
	\delta_y \revgausem{h}(x) \propto \exp \tond*{-U(x)-\frac{\abs*{x-y}^2}{2h}}.
	\]
    It turns out that even for the Proximal Sampler there is a strong relation between the log-concavity of $\pi$ and the curvature of the process: in particular, if $\pi$ is $\alpha$-log-concave then \eqref{eq:p-coarse-ricci-geq-K} holds with 
    $K = \frac{1}{\alpha h+1}$.
    As for the corresponding powerful consequences of having positive curvature, a reverse transport-entropy regularization for the Proximal Sampler under log-concavity was proved in \cite{che-che-sal-wib-2022}, while a local Poincaré inequality for the iterates of the algorithm was established in \cite{ped-sal-2025-bis}.

    \subsection{Cutoff phenomenon}
    Among the many applications of non-negative curvature and its classical consequences, in recent years this property has led to a better understanding of a fascinating phase transition, known as the cutoff phenomenon,  observed for many Markov processes in the limit as their size tends to infinity. We refer the reader to \cite{sal-2025-modern-aspects} for a modern introduction and we discuss here only the main ingredients. 
    Suppose that our Markov chain is ergodic with invariant measure $\pi$, and recall the total variation between $\mu$ and $\pi$, as a measure of distance from equilibrium:
    \begin{equation*}
    \totvar{\mu}{\pi}  \coloneqq   \sup_{A\in \mathcal{B}(\Omega)}\abs*{\mu(A)-\pi(A)}.
    \end{equation*} 
    Correspondingly, the \emph{mixing time}
    of the process $(X_k)_{k\ge 0}$ with initialization $\mu_0$  and precision $\varepsilon\in(0,1)$ is defined as
 \begin{equation*}
\label{eq:mix-time}
 \tmix(\mu_0,\varepsilon) \coloneqq  \inf\{k\ge 0\colon \totvar{\mu_k}{\pi}\le\varepsilon \}.
 \end{equation*} 
 The cutoff phenomenon is the observation that, in many situations, for any fixed $\eps \in (0,\frac12)$ the width of the mixing window
 \begin{equation*}
     \wmix(\mu_0, \eps) \coloneqq \tmixa{\mu_0,\eps} -\tmixa{\mu_0,1-\eps}
 \end{equation*}
    is asymptotically much smaller than the mixing time $\tmixa{\eps}$, in the limit as the size of the system described by the Markov chain goes to infinity.
    To be rigorous, this means that in the study of cutoff $\mu_0 = \mu_0^{(n)}$, $\pi = \pi^{(n)}$, and the Markov chain $(X_k^{(n)})_k$ implicitly depend on a size-parameter $n\in \N$, although this dependence is not explicitly written to ease the notation; correspondingly, cutoff is then said to occur if for all $\eps\in(0,\frac12)$
    \begin{equation*}
        \frac{\wmixa{\eps}}{\tmixa{\eps}} \to 0 \qquad \text{ as } n \to \infty.
    \end{equation*}
    The cutoff phenomenon was initially discovered by Aldous, Diaconis and Shahshahani in the context of Markov chains describing card shuffling methods \cite{ald-1983,dia-1996, dia-sha-1981, ald-dia-1986}, and since then has been established for a large variety of different models.
    In spite of the abundance of examples, however, proving cutoff is traditionally a difficult task, because it is complicated to compute the mixing time exactly. More recently, however, a different approach has emerged \cite{sal-2023, sal-2025-modern-aspects}, which aims to provide directly upper bounds on the mixing-window $\wmix(\eps)$ in terms of fundamental statistics and/or parameters of the Markov chain $(X_k)_k$ and its stationary distributions $\pi$.
    For example, \cite{sal-2025-diffusions} proved for the Langevin dynamics that 
    when $\pi$ is $\alpha$-log-concave and $X_0 = x$ is deterministic then
    \begin{equation}\label{eq:mix-wind-non-neg-curv-diffusions}
        \wmix(\delta_x, \eps) \lesiep
        \begin{cases}
        \sqrt{\cpi{\pi} \tmixa{\delta_x,1-\eps}}+\cpi{\pi} &\quad \text{ if } \alpha =0,
        \\
        \frac{1}{\alpha} &\quad \text{ if } \alpha >0,
        \end{cases}
    \end{equation}
    where $\cpi{\pi}$ is the best constant in the Poincaré inequality for $\pi$,
    \[
      \variwr{\pi}{f} \leq \cpi{\pi} \int \abs*{\nabla f}^2 \dd \pi \qquad \text{ for all } f\in C^1(\R^d),
    \]
    and where we use the symbol $\lesiep$ to denote inequalities up to positive multiplicative constants depending only on $\eps$.
    Analogous bounds were then proved also for the Proximal Sampler \cite{ped-sal-2025-bis} and for discrete Markov chains with non-negative Bakry--\'Emery curvature \cite{ped-sal-2025}.
    Besides being interesting on their own, one major advantage of estimates like \eqref{eq:mix-wind-non-neg-curv-diffusions} is that they lead to easily verifiable sufficient conditions for the occurrence of the cutoff phenomenon:  for example, \eqref{eq:mix-wind-non-neg-curv-diffusions} for $\alpha =0$ implies that cutoff happens as soon as  
    \begin{equation}
        \frac{\tmixa{\delta_x, \eps}}{\cpi{\pi}} \to \infty,
    \end{equation}
    which is known as the \emph{product condition} in the theory of mixing times, and was famously conjectured by Peres to be a sufficient condition in predicting cutoff for most ``reasonable'' examples of Markov chains.
    So far, however, fruitful estimates like \eqref{eq:mix-wind-non-neg-curv-diffusions} have been established only in non-negatively curved settings because the arguments crucially rely on some classical consequences of that, notably the validity of the local Poincaré inequality \eqref{eq:local-poincare-curved} in the case of \cite{sal-2023,sal-2025-diffusions,ped-sal-2025}, and the regularization estimate \eqref{eq:KL-W-log-conc-reg} in \cite{ped-sal-2025-bis}.
    
    \subsection{Negatively curved Markov chains}
    Having non-negative curvature is, however, a restrictive condition for a Markov chain, and the main motivation of this work is to investigate whether it is possible to relax Ollivier's condition \eqref{eq:p-coarse-ricci-geq-K} while still deriving analogous consequences.
    In particular, our focus in this work is on Markov chains satisfying Assumption \ref{ass:weak-curv-wp} below: this assumption was also briefly discussed in Ollivier's paper \cite{oll-2009}, where it was referred to as ``positive curvature up to $M$''. While the motivation in \cite{oll-2009} was topological, i.e., to turn positive curvature into an \emph{open} property, in his subsequent survey Ollivier explicitly raised the question of understanding which consequences of \eqref{eq:p-coarse-ricci-geq-K} extend to the more general setting of \eqref{eq:weak-corse-ricci-curvature} below, see \cite[Problem O]{oll-2010}.
    \begin{assumption}\label{ass:weak-curv-wp}
		There exist constants  $K\in [0,1],M \ge 0$ and $p\ge 1$ such that for all $x,y\in \Omega$
		\begin{equation}
				\label{eq:weak-corse-ricci-curvature}
			W_p\tond*{\delta_x P, \delta_y P} \le K \dist(x,y) +M.
		\end{equation}
	\end{assumption}
    \begin{remark}
        As for the coarse Ricci curvature, it is not difficult to check that the above implies in fact that for all $\mu,\nu \in \pp_p\tond*{\Omega}$ and $N\ge 0$
        \[
            W_p \tond*{\mu P^N, \nu P^N} \le K^N W_p(\mu,\nu) + M \frac{1-K^N}{1-K},
        \]  
        where we interpret $\frac{1-K^N}{1-K} = N$ if $K=1$.
    \end{remark}
    While for the Langevin dynamics and the Proximal Sampler non-negative curvature was intimately related to the log-concavity of $\pi$, it turns out that the condition naturally corresponding to \eqref{eq:weak-corse-ricci-curvature} is the following perturbative one.
	\begin{assumption}\label{ass:log-lip-pert-log-conc}
		The potential $U$ decomposes as 
		$	U = V+H$
		where
		$V,H\in C^2(\R^d)$ are such that for some constants $\alpha, L \ge 0$ we have  $\hess V \succcurlyeq \alpha I_d$ and $\abs*{\nabla H} \le L$.
	\end{assumption}
    The family of distributions satisfying Assumption \ref{ass:log-lip-pert-log-conc} is considerably larger than the one of log-concave measures. For example, it includes: (i) potentials $U\in C^2(\R^d)$ whose Hessian $\hess U(x)$ is positive definite only outside of a bounded subset of $\R^d$, and (ii) convolutions of bounded support distributions with a Gaussian measure (cf. the arXiv version of \cite{bri-ped-2024}).

    In order to work with the Euler--Maruyama discretization \eqref{eq:LMC} we also consider the following Assumption for simplicity, which becomes only qualitative in the limit as the step size $h\to 0$ (i.e. $\beta$ does not enter in the quantitative bounds for the Langevin dynamics \eqref{eq:LD}).
	\begin{assumption}\label{ass:log-lip-potential}
		There exists a constant $\beta>0$ such that $\hess V \preccurlyeq \beta I_d$.
	\end{assumption}
    As promised, in the next lemma we discuss the validity of \eqref{eq:weak-corse-ricci-curvature} under the above assumptions.
    \begin{lemma} \label{lem:LMC-PS-weak-curvature}
    Under Assumption \ref{ass:log-lip-pert-log-conc}, for $p\ge 1$ and $\mu,\nu \in \pp_p\tond*{\R^d}$ the following hold.
    \begin{enumerate}
        \item If additionally Assumption \ref{ass:log-lip-potential} holds and $h\le \frac1\beta$, then for the LMC transition kernel
        \begin{equation}
            W_p\tond*{\mu \lmcsem{h}, \nu \lmcsem{h}} \le (1-\alpha h)W_p(\mu,\nu) + 2Lh.
        \end{equation}
        \item For the Proximal Sampler kernel
        \begin{equation}
            W_p\tond*{\mu \pssem{h}, \nu \pssem{h}} \leq \frac{1}{\alpha h + 1}W_p(\mu,\nu) +\frac{2Lh}{\alpha h+1}.
        \end{equation}
    \end{enumerate}
    \end{lemma}
    \begin{proof}
        \begin{enumerate}
            \item It suffices to observe that if $h\le \frac1\beta$
        \begin{equation*}
        \begin{split}
            & \abs*{x -h \nabla V (x) -h\nabla H(x) -\tond*{ y -h \nabla V (y) -h\nabla H(y)}} 
            \\
            \le & \abs*{x -h \nabla V (x)  -\tond*{ y -h \nabla V (y) }} +2hL
            \\
             \le & (1-\alpha h) \abs*{x-y}+2hL 
        \end{split}
        \end{equation*}
        where the last line uses the well-known contractivity of gradient descent for strongly convex functionals and small enough step-size, cf. e.g. \cite[Lemma 2.2]{alt-tal-2022}.
        \item The forward step is clearly non-negatively curved, i.e. \eqref{eq:p-coarse-ricci-geq-K} holds with $K=1$ for $\gausem{h}$.
		For the backward step, fix $y_1,y_2 \in \R^d$ and consider the 
		probability density 
		\[
				\mu(x) \propto \expo{-V(x) -\frac{\abs*{x-y_1}^2}{2h}}
		\]
		which is $(\alpha + \frac1h)$-log-concave.
		By \cite[Corollary 2.3]{khu-maa-ped-2024} we have
		\begin{equation} \label{eq:curv-rev-step}
        \begin{split}
            	W_\infty\tond*{\delta_{y_1}\revgausem{h},\delta_{y_2}\revgausem{h}} & \leq W_\infty\tond*{\delta_{y_1}\revgausem{h},\mu} + W_\infty\tond*{\mu,\delta_{y_2}\revgausem{h}}
			\\
			& \leq \frac{L}{\alpha+\frac1h}+ \frac{L+\frac{\abs*{y_1-y_2}}{h}}{\alpha+\frac1h},
        \end{split}
		\end{equation}
		and the conclusion easily follows.
        \end{enumerate}
    \end{proof}

    \subsection{Main results and paper organization}
    The first question that we study in Section \ref{sec:def-loc-conc} is concerned with the concentration of measure properties of the iterates $\mu_k = \law(X_k)$ of the Markov chain at a \emph{finite-time}, in the spirit of the local Poincaré inequality \eqref{eq:local-poincare-curved}, but under the weaker Assumption \ref{ass:weak-curv-wp}.
    It turns out that, under the corresponding assumption on the initialization $\mu_0$ and the one-step transition measures $\delta_x P$, the distribution $\mu_k$ satisfies a defective version of Talagrand's transport-entropy inequality, cf. Theorem \ref{thm:local-defective-transport-entropy}. For example, for the Langevin dynamics under Assumption \ref{ass:log-lip-pert-log-conc}, we deduce that for all $\mu \in \pp \tond*{\R^d}, x\in \R^d, t>0$
    \[
        W_2(\mu,\delta_x \ldsem_t) \leq \sqrt{ 2\frac{1-e^{-2\alpha t}}{\alpha}\kldiv{\mu}{\delta_x \ldsem_t}} + \frac{2L\tond*{1-e^{-\alpha t}}}{\alpha}:
    \]
    this is a standard transport-entropy inequality when $L=0$ (corresponding to the positive curvature settings), and includes an extra additive cost in the general case $L>0$.
    
    The second problem that we address is to deduce an analogue of the reverse transport-entropy regularization \eqref{eq:KL-W-log-conc-reg} in the weaker setting of Assumption \ref{ass:weak-curv-wp}. 
    We specialize the discussion to Euclidean dynamics and the case $p=2$: again under the appropriate assumption on the one-step transition measures, we establish a similar regularization effect for the relative entropy, see Theorem \ref{thm:reverse-transport-entropy-discrete-time}.
    For example, again in the simplest case of the Langevin dynamics under Assumption \ref{ass:log-lip-pert-log-conc}, Corollary \ref{cor:kl-w-reg-LD} gives that 
    \begin{equation}
			\kldiv{\mu\ldsem_T}{\nu\ldsem_T} \leq \begin{cases}
				& \frac{\tond*{W_2(\mu,\nu) +2LT}^2}{4T} \quad \text{ if } \alpha =0,
				\\
				&\frac{\alpha}{2\tond*{1-e^{-2\alpha T}}}
				\tond*{e^{-\alpha T}W_2(\mu,\nu)+\frac{2L}{\alpha}\bigl(1-e^{-\alpha T}\bigr)}^2 \quad \text{ if } \alpha>0.
			\end{cases}
		\end{equation}

    Finally, adapting the high-level approach of \cite{ped-sal-2025-bis}, in Section \ref{sec:cutoff-criteria} we apply our results to deduce estimates on the mixing windows and cutoff criteria for the Langevin dynamics and the Proximal Sampler, valid in the negatively curved setting of Assumption \ref{ass:log-lip-pert-log-conc} with $\alpha >0$.
    For example, for \eqref{eq:LD} with deterministic initialization we prove that 
    \begin{equation}
    \begin{split}
            \wmixa{\eps} & \lesiep \cpi{\pi}\tond*{\frac{L^2}{\alpha}+1}+\sqrt{\cpi{\pi}}\tond*{\frac1{\sqrt{\alpha}}+\frac{L}{\alpha}}
            \\
            & \lesiep \expo{\frac{L^2}{\alpha}+\frac{4L}{\sqrt{\alpha}}}\cdot \frac1\alpha\tond*{1+\frac{L^2}{\alpha}+\frac{L}{\sqrt{\alpha}}}.
    \end{split}
    \end{equation}
    Notice in particular that this recovers the second estimate in \eqref{eq:mix-wind-non-neg-curv-diffusions} in the case $L=0<\alpha$, but the case $L>0$ covers a significantly more general setting. 
    Furthermore, in Theorem \ref{thm:cutoff-LD-general} we also enlarge the scope of admissible initializations compared to \cite{sal-2025-diffusions, ped-sal-2025-bis}, which were restricted to Dirac masses and measures $\mu_0$ satisfying a Poincaré inequality: instead, we consider initial measures satisfying a defective transport-entropy inequality, which in particular does not require that $\mu_0$ has connected support.
    While for cutoff applications for concreteness we focus on our two usual examples, we stress that the same approach works more generally for Markov chains to which Theorems \ref{thm:local-defective-transport-entropy} and \ref{thm:reverse-transport-entropy-discrete-time} apply, with straightforward modifications.

    \section{Defective local concentration}
    \label{sec:def-loc-conc}
	In this section, we work in the general setting of  a Polish space $\tond*{\Omega, \dd}$ equipped with its Borel $\sigma$-field.  For $1\le p \le 2$ and a Borel probability measure $\pi \in \pp_p\tond*{\Omega}$, we say  that $\pi$ satisfies a \emph{defective p-transport-entropy inequality} with constants $A, B\ge 0$ (notation: 		$\pi \in \deftpha{A,B}$)  if for all $\nu\in \pp\tond*{\Omega}$
	\begin{equation}
		W_p\tond*{\nu,\pi} \leq \sqrt{2A \kldiv{\nu}{\pi}}+B.
	\end{equation}
		When $B = 0$ in the above, $\pi$ satisfies the usual \emph{$p$-transport-entropy inequality} (notation: $\pi \in \tpha{A}$).
		A few sufficient conditions are known for the validity of $\tpha{C}$, cf. e.g. \cite{ott-vil-2000, dje-gui-wu-2004}, but the defective form of the inequality has received less attention.
		As a warmup and to provide some intuition, we give first a simple sufficient condition in Lemma \ref{lem:def-tth-winfty-close}.
        This is already useful in some situations: for example, if Assumption \ref{ass:log-lip-pert-log-conc} holds with $\alpha >0$, it follows that $\pi\in \defttha{\frac1\alpha, \frac{2L}{\alpha}}$, as it can be seen
        by considering the  $\alpha$-log-concave measure 
        $\mu \propto e^{-V}$, which is such that $\mu\in \defttha{\frac1\alpha}$ and that  $W_\infty(\mu,\pi) \le \frac{L}{\alpha}$ by \cite[Proposition 1.1]{khu-maa-ped-2024}.
		\begin{lemma} \label{lem:def-tth-winfty-close}
			Let $\mu,\pi\in \pp_p\tond*{\Omega}$ be such that $\mu\in\deftpha{C,M}$ for some constants $C,M \ge 0$ and $W_\infty(\mu,\pi)<\infty$.
			Then, $\pi \in\deftpha{C,M+2W_\infty(\mu,\pi)}$.
		\end{lemma}
		\begin{proof}
			Let $\nu \in \pp\tond*{\Omega}$ such that $\kldiv{\nu}{\pi}<\infty$ and
			$\gamma \in \pp \tond*{\Omega\times \Omega}$ an optimal coupling for $W_\infty\tond*{\pi,\mu}$.
			By the disintegration theorem applied to $\gamma$ \cite{amb-gig-sav-2008}, we can find a  collection of probability measures $\tond*{\mu^x}_{x\in \Omega}$ on $\Omega$ such that (with abuse of notation) $d\gamma(x, y) = d\pi(x)d\mu^x(y)$.
			Consider now the probability measure $\tilde \gamma \in \pp\tond*{\Omega \times \Omega}$ defined by
            $d\tilde \gamma (x,y) = d\nu(x)d\mu^x(y)$.
            Observe that the first marginal of $\tilde \gamma$ is $\nu$, and denote also by $\tilde \nu$ its second marginal. By construction, 
			\[
			W_\infty(\nu,\tilde \nu) \le W_\infty(\mu, \pi). 
			\]
			Furthermore, by the data-processing inequality for the relative entropy 
            \[
            \kldiv{\tilde \nu}{\mu} \le \kldiv{\tilde \gamma}{\gamma} = \kldiv{\nu}{\pi}.
            \]
			Hence, it follows that
			\[
			W_p\tond*{\tilde \nu, \mu} \leq \sqrt{2C \kldiv{\tilde \nu}{\mu}}+M \leq  \sqrt{2C \kldiv{ \nu}{\pi}}+M.
			\]
			The conclusion follows from the triangle inequality and the bound $W_p\le W_\infty$, since
			\begin{align*}
				W_p\tond*{\nu,\pi} \le 			W_\infty\tond*{\nu,\tilde \nu} +			W_p\tond*{\tilde \nu,\mu} +			W_\infty\tond*{\mu,\pi} .
			\end{align*}
		\end{proof}
        \begin{remark}
            This lemma already highlights some important differences between the defective form of the inequality and the standard one, when $p>1$. To see  this, consider for $\eps>0$ the probability measures $\pi_\eps,\mu \in \pp(\R)$ given by $\pi_\eps = \frac12(\delta_{-\eps}+\delta_\eps)$ and $\mu = \delta_0$. Then, $W_\infty\tond*{\mu,\pi_\eps}\le \eps$ and $\mu \in \deftpha{0}$ for every $p\ge 1$, and so by Lemma \ref{lem:def-tth-winfty-close} we deduce that $\pi_\eps \in \deftpha{0,2\eps}$. However, for any $C>0$ and $p>1$ we necessarily have that $\pi_\eps \notin \deftpha{2C}$, since the support of $\pi_\eps$ is disconnected (cf. \cite[Lemma 2.1]{fat-shu-2018}, whose argument is easily adapted to the general case $p>1$); similarly, $\pi_\eps$ does also not satisfy other classical functional inequalities, such as the Poincaré inequality.
            In many situations, including in the study of the cutoff phenomenon in Section \ref{sec:cutoff-criteria}, it is still desirable to think of the measure $\pi_\eps$ as being ``well-concentrated'' when $\eps$ is small, and one important motivation for the introduction of the defective transport-entropy inequalities is to rigorously capture this intuition, despite the failure of some classical functional inequalities.
        \end{remark}
     
        Our focus in this work is more generally on local functional inequalities and the relationship with Assumption \ref{ass:weak-curv-wp}, and in this sense, our main result in this section is   Theorem \ref{thm:local-defective-transport-entropy} below.
        Inspired by earlier work of Marton \cite{mar-1996} and the subsequent developments in \cite{dje-gui-wu-2004, eld-lee-leh-2017}, which were, however, concerned with the classical non-defective transport inequalities, we make the key assumption that the one-step transition kernel satisfies a defective transport-entropy inequality:
        \begin{assumption} \label{ass:one-step-def-tth}
			There exist constants $A,B \ge 0$ and $1\le p \le 2$ such that for all $x\in \Omega$
			\[
			\delta_x P \in \deftpha{A,B}.
			\]
		\end{assumption}
        It turns out that the above, combined with our defective curvature Assumption \ref{ass:weak-curv-wp}, allows us to track the propagation of a defective transport-entropy inequality through the action of the Markov kernel $P$.
		\begin{theorem} \label{thm:local-defective-transport-entropy}
			Under Assumptions \ref{ass:weak-curv-wp} and \ref{ass:one-step-def-tth} (for the same $p$), if $\mu_0 \in \deftpha{J,S}$ then 
			\[
					\mu_0P \in \deftpha{A+K^2J, KS +B+M}.
			\]
			In particular, for $N\ge 1$ by iterating
			\[
					\mu_0 P^N \in \deftpha{K^{2N}J + A\frac{1-K^{2N}}{1-K^2}, K^{N}S+ (B+M)\frac{1-K^{N}}{1-K}},
			\]
            where as usual  we interpret $\frac{1-K^{N}}{1-K} = N$ if $K=1$.
		\end{theorem}
		\begin{proof}
		Write $\mu_1 = \mu_0 P$ and 
		let $ \nu_1 \in \pp\tond*{\Omega}$ such that 
		\[
				\kldiv{ \nu_1}{ \mu_1}<\infty.
		\]
		We need to prove that 
		\[
				W_p\tond*{\nu_1,\mu_1} \leq \sqrt{2(A+K^2J)\kldiv{\nu_1}{\mu_1}} +KS +B+M.
		\]
		Let $\pmb{\mu} \in \pp \tond*{\Omega\times \Omega}$ be the joint law of $(X_0,X_1)$ with $X_0 \sim \mu_0, X_1 \sim \delta_{X_0}P$ (i.e., with abuse of notation, $d\mubf (x,y) = d\mu_0(x)d(\delta_xP)(y)$), so that its marginals are $\mu_0, \mu_1$.
		Define $\pmb{\nu} \in \pp \tond*{\Omega\times \Omega}$ by
		\[
				\relden{\pmb{\nu}}{\pmb{\mu}}(x,y) = \relden{\nu_1}{\mu_1}(y),
		\]
		so that the second marginal of $\pmb{\nu}$ is $\nu_1$ and
		\[
				\kldiv{\pmb{\nu}}{\pmb{\mu}} = \kldiv{\nu_1}{\mu_1}.
		\]
		We denote by $\nu_0$ the first marginal of $\pmb{\nu}$ and by the disintegration theorem
		we write $d\pmb{\nu}(x,y) = d\nu_0(x) d\nu_1^{x}(y)$ for $\nu_1^x \in \pp \tond*{\Omega}$.
		Observe that by the chain rule for the relative entropy
		\begin{equation} \label{eq:kl-decomposition-fol-chain-rule}
			\begin{split}
				\kldiv{\nu_1}{\mu_1}  = \kldiv{\pmb{\nu}}{\pmb{\mu}}
				= \kldiv{\nu_0}{\mu_0} + \int_{\Omega} \kldiv{\nu_1^x}{\delta_xP} \dd \nu_0(x).
			\end{split}
		\end{equation}		
		We construct a coupling $(Y_1,X_1)$ of $(\nu_1, \mu_1)$ as follows.
		First, consider random variables $(Y_0,X_0) \sim (\nu_0,\mu_0)$ optimally coupled for $W_p(\nu_0,\mu_0)$; then, conditionally on $Y_0,X_0$, sample 
		$( Y_1, X_1)$ optimally coupled  for $W_p\tond*{\nu_1^{Y_0},\delta_{X_0}P}$, in a measurable way (which is possible by \cite[Corollary 5.22]{vil-2009}).
		We then have that
		\begin{equation}
			\begin{split}
				& W_p\tond*{\nu_1,\mu_1} \leq \graf*{\expe{\dist\tond*{Y_1,X_1}^p}}^\frac1p
				\\
				 = & \graf*{\expe{\condexpe{\dist\tond*{Y_1,X_1}^p}{X_0,Y_0}}}^\frac1p
				\\
				 = &\graf*{\expe{W_p^p\tond*{\nu_1^{Y_0}, \delta_{X_0}P}}}^\frac1p
				\\
				\leq &  \graf*{\expe{{W_p^p\tond*{\nu_1^{Y_0}, \delta_{Y_0}P} }}}^\frac1p + \graf*{\expe{W_p^p\tond*{\delta_{Y_0}P, \delta_{X_0}P}}}^\frac1p \quad \text{ by the triangle inequality}
				\\
				\leq & \graf*{\expe{\tond*{\sqrt{2A \kldiv{\nu_1^{Y_0}}{\delta_{Y_0}P}} +B}^{p} }}^\frac1p+\graf*{\expe{\tond*{K\dist\tond*{X_0,Y_0}+M}^p}}^\frac1p 	\quad \text{ by Assumptions \ref{ass:one-step-def-tth} and \ref{ass:weak-curv-wp} }
				\\
				 \leq& \sqrt{2A\expe{\kldiv{\nu_1^{Y_0}}{\delta_{Y_0}P}}}+KW_p(\nu_0,\mu_0)+B+M \quad \text{ by Jensen's and triangle inequality}
				\\
				\leq & \sqrt{2A\expe{\kldiv{\nu_1^{Y_0}}{\delta_{Y_0}P}}}+K\graf*{\sqrt{2J\kldiv{\nu_0}{\mu_0}}+S}+B+M \quad \text{ since } \mu_0 \in \deftpha{J,S}
				\\
				 \leq & \sqrt{2(A+K^2J)\graf*{\kldiv{\nu_0}{\mu_0}+\expe{\kldiv{\nu_1^{Y_0}}{\delta_{Y_0}P}}}}+KS +B+M   \quad \text{ by Cauchy--Schwarz }
				\\
				 = &\sqrt{2(A+K^2J)\kldiv{\nubf}{\mubf}} +KS +B+M  \quad \text{ by \eqref{eq:kl-decomposition-fol-chain-rule} }
				\\
				 = &\sqrt{2(A+K^2J)\kldiv{\nu_1}{\mu_1}} +KS +B+M.
			\end{split}
		\end{equation}
        This concludes the proof of the theorem.
		\end{proof}

        We now apply Theorem \ref{thm:local-defective-transport-entropy} to our two model examples, the Langevin dynamics and the Proximal Sampler.
        Before proceeding, we also state the following useful lemma, which we will use later, and whose proof is analogous to the one of \cite[Lemma 2.2]{dje-gui-wu-2004}.
        \begin{lemma}\label{lem:weak-conv-def-tph}
            Consider a sequence of probability measures $(\mu_n)_n$ in $\pp_p(\Omega)$ such that for some constants $A,B\ge 0 $ we have that $\mu_n \in \deftpha{A,B}$.
            If $\mu_n$ converges weakly to $\mu \in \pp_p\tond*{\Omega}$ as $n \to \infty$, then 
            $\mu \in \deftpha{A,B}$.
        \end{lemma}

		\begin{corollary}\label{cor:local-defective-transport-entropy-LD}
			Suppose $\mu_0 \in \defttha{J,S}$. Under Assumptions \ref{ass:log-lip-pert-log-conc} and \ref{ass:log-lip-potential}, for $T>0$ we have
			\[
					\mu_0 \ldsem_T \in \begin{cases}
						\defttha{J+2T, S+2LT} \quad \text{ if } \alpha =0,
						\\
						\defttha{e^{-2\alpha T} J+\frac{1-e^{-2\alpha T}}{\alpha}, e^{-\alpha T}S+2L\frac{1-e^{-\alpha T}}{\alpha} } \quad \text{ if } \alpha >0.
					\end{cases}
			\]
			
		\end{corollary}
		\begin{proof}
			We work in discrete time, with $N\ge \beta T$  and step-size $h = \frac{T}{N}$ for \eqref{eq:LMC} and we prove the bound in the case $\alpha >0$ (the case $\alpha =0$ follows similarly). 
			Then, since $\delta_x \lmcsem{h} = \mathcal{N}(x-h\nabla U(x), 2h I_d)$ is $\frac{1}{2h}$-log-concave, Assumption \ref{ass:one-step-def-tth} is satisfied with 
			\[
					A = 2h, \quad  B = 0.
			\]
			Applying Theorem \ref{thm:local-defective-transport-entropy} (recall Lemma \ref{lem:LMC-PS-weak-curvature}), we deduce that 
			\[
					\mu_0 \tond*{\lmcsem{h}}^N \in \defttha{A_{T,N}, B_{T,N}}
			\]
			with 
			\begin{align}
				A_{T,N} = \tond*{1-\alpha\frac{T}{N}}^{2N}J + 2\frac{T}{N}\frac{1-\tond*{1-\alpha\frac{T}{N}}^{2N}}{1-\tond*{1-\alpha\frac{T}{N}}^2},
				\\
				B_{T,N} = \tond*{1-\alpha\frac{T}{N}}^{N}S+ 	2L\frac{T}{N}\frac{1-\tond*{1-\alpha\frac{T}{N}}^{N}}{1-\tond*{1-\alpha\frac{T}{N}}}.
			\end{align}
			The conclusion follows by letting $N\to \infty$ by Lemma \ref{lem:weak-conv-def-tph}, since $\mu_0 \tond*{\lmcsem{h}}^N \to \mu_0\ldsem_T$ weakly under Assumption \ref{ass:log-lip-pert-log-conc} and \ref{ass:log-lip-potential}, as can be checked from \cite[Theorem 3.4]{mao-2008} and \cite[Theorem D]{kan-nak-1988}.
			\end{proof}
			
			\begin{corollary} \label{cor:local-defective-th-proximal-sampler}
				Suppose $\mu_0 \in \defttha{J,S}$ and that Assumption \ref{ass:log-lip-pert-log-conc} holds, and let $N\in \N_{\ge 1}$.
                If $\alpha =0$, we have
                \begin{equation*}
                    \mu_0 \tond*{\pssem{h}}^N \in \defttha{J+2Nh, S+6NLh},
                \end{equation*}
                and if $\alpha >0$ we have
                \begin{equation*}
                    \mu_0 \tond*{\pssem{h}}^N \in \defttha{\frac{J}{(1+\alpha h)^{2N}}+\frac{1}{\alpha}\tond*{1-\frac{1}{(1+\alpha h)^{2N}}},\frac{S}{(1+\alpha h)^{N}}+\frac{6L}{\alpha}\tond*{1-\frac{1}{(1+\alpha h)^{N}}}}.
                \end{equation*}
			\end{corollary}

            \begin{proof}
                Fix $x \in \R^d$. We prove the bound in the case $\alpha >0$ (the case $\alpha =0$ follows similarly). Observe first that 
                $\delta_x\gausem{h} = \mathcal{N}(x,hI_d) \in \defttha{h,0}$.
                Next, for the backward step, we claim that
                \[
                    \delta_x \revgausem{h}(y) \propto \expo{-V(y)-H(y) -\frac{\abs*{y-x}^2}{2h}} \in \defttha{\frac{h}{\alpha h +1}, \frac{2Lh}{\alpha h +1}}.
                \]
                Indeed the probability measure $\nu(y) \propto \expo{-V(y) - \frac{\abs*{y-x}^2}{2h}}$ is such that
                \begin{align}
                    &\nu \in \defttha{\frac{h}{\alpha h +1},0} & \quad \text{ by strong log-concavity},
                    \\
                    & W_\infty(\delta_x \revgausem{h}, \nu) \le \frac{Lh}{\alpha h +1} &\quad\text{ by  \cite[Corollary 2.3]{khu-maa-ped-2024}},
                \end{align}
                and so the claim follows from Lemma \ref{lem:def-tth-winfty-close}.
                Remembering \eqref{eq:curv-rev-step}, we deduce from Theorem \ref{thm:local-defective-transport-entropy} with $\mu_0 = \delta_x \gausem{h}$ and $P = \revgausem{h}$ that
                \[
                        \delta_x \pssem{h}= \delta_x \gausem{h} \revgausem{h} \in \defttha{\frac{h(2+\alpha h)}{(1+\alpha h)^2}, \frac{4Lh}{1+\alpha h}}.
                \]
               Recalling again Lemma \ref{lem:LMC-PS-weak-curvature} we can apply Theorem \ref{thm:local-defective-transport-entropy} to  $P = \pssem{h}$ to deduce that
                \[
                        \mu_0 \tond*{\pssem{h}}^N \in \defttha{\frac{J}{(1+\alpha h)^{2N}}+\frac{1}{\alpha}\tond*{1-\frac{1}{(1+\alpha h)^{2N}}},\frac{S}{(1+\alpha h)^{N}}+\frac{6L}{\alpha}\tond*{1-\frac{1}{(1+\alpha h)^{N}}}}.
                \]
                
            \end{proof}

	\section{Entropy-Wasserstein regularization}
    \label{sec:kl-w-reg}
    In this section, we abandon the general Polish-space setting to focus on Markov kernels on the space $\R^d$, equipped with the Euclidean distance.
    The main goal here is then to investigate whether it is possible to derive reverse transport-entropy inequalities under Assumption \ref{ass:weak-curv-wp}.
    Our analysis is based on the framework developed in \cite{alt-che-2023} (see also the subsequent works \cite{alt-che-2025, alt-che-III-2024, alt-che-zha-2025}), as a generalization of the \emph{Privacy Amplification by Iteration} technique originating in the literature of differential privacy \cite{fel-mir-tal-tha-2018, alt-tal-2022}.
    In particular, following \cite{alt-che-2023}, we consider the following assumption, stating that a reverse transport-entropy inequality holds for the one-step transition kernel $P$.
    	\begin{assumption}\label{ass:kl-reg-one-step}
		There exists a constant $C \ge 0$ such that for all $x,y\in \R^d$
		\begin{equation} \label{eq:kl-w-reg-one-step}
			\kldiv{\delta_x P}{\delta_y P} \leq C{\abs*{x-y}}^2.
		\end{equation}
	\end{assumption}
    The point of the above is that this assumption is often easy to check and holds in some interesting situations, in particular, as we will see, for many stochastic dynamics involving Gaussian noise (as expected from the well-known identity  $\kldiv{\mathcal{N}(x,\sigma^2)}{\mathcal{N}(y,\sigma^2)} = \frac{\abs{x-y}^2}{2\sigma^2}$ for the relative entropy between Gaussian measures with different means).
    The main contribution of \cite{alt-che-2023} was then to show that a lower bound on the coarse Ricci curvature as in \eqref{eq:p-coarse-ricci-geq-K} allows one to propagate \eqref{eq:kl-w-reg-one-step} in a non-trivial way, giving a regularization effect for the iterated kernel $P^N$.
    Our main result in this section is Theorem \ref{thm:reverse-transport-entropy-discrete-time} below, which states that even under the weaker Assumption \ref{ass:weak-curv-wp} it is still possible to deduce a remarkable regularization effect for $P^N$, again in the form of a reverse transport-entropy inequality.
    The proof, as in \cite{alt-che-2023}, is based on coupling two Markov chains $(X_n)_{n=0}^N, (Y_n)_{n=0}^N$ evolving according to the same kernel $P$ but having different initializations $X_0\sim \mu, Y_0 \sim \nu$, and on constructing a third interpolating process $(X'_i)_{i=1}^N$ such that $X'_0 = Y_0 $ and $X'_N = X_N$.
    By the data-processing inequality we have then that (slightly abusing notation)
    \[
        \kldiv{\mu P^N}{ \nu P^N} = \kldiv{X'_N}{Y_N} \leq \kldiv{\tond*{X'_n}_{n=0}^N}{\tond*{Y_n}_{n=0}^N}.
    \]
    The right-hand side is then controlled with the so-called \emph{shifted composition rule}, a variation of the classical chain rule for the relative entropy that exploits the convexity properties of this divergence.
    The difference from \cite{alt-che-2023} is that instead of relying on \eqref{eq:p-coarse-ricci-geq-K} in the design and study of the interpolating process we only use the weaker \eqref{eq:weak-corse-ricci-curvature}, and this later leads us to solving a more involved optimization problem to  choose in the best way the ``shifting parameters'' appearing in the construction.
	\begin{theorem}\label{thm:reverse-transport-entropy-discrete-time}
			Suppose Assumptions \ref{ass:kl-reg-one-step} and \ref{ass:weak-curv-wp} for $p=2$ hold, and let $\mu,\nu\in \pp_2\tond*{\R^d}$ and $N\in \N_{\ge 1}$. 
			Then, if $K=1$
			\begin{equation} \label{eq:rev-w-h-disc-non-neg-curv-weak}
            \begin{split}
                \kldiv{\mu P^N}{\nu P^N} &\leq C\cdot\begin{cases}
					\frac{\tond*{W_2(\mu,\nu)+(N-1)M}^2}{N} \; &\text{ if } W_2(\mu,\nu)\ge M
					\\
					W_2(\mu,\nu)^2+(N-1)M^2 \; &\text{ if } W_2(\mu,\nu)< M
				\end{cases}
                \\
                & \leq  2C\frac{\tond*{W_2(\mu,\nu)+(N-1)M}^2}{N}.
            \end{split}
			\end{equation}
			If $0\le K<1$, for $N\ge 2$
			\begin{equation} \label{eq:rev-w-h-disc-pos-curv-weak}
            \begin{split}
                				\kldiv{\mu P^N}{\nu P^N} &\le C\cdot \begin{cases}
					\frac{1-K^2}{1-K^{2N}} \tond*{K^{N-1}W_2(\mu,\nu)+M\frac{1-K^{N-1}}{1-K}}^2 \; &\text{if } W_2(\mu,\nu)\ge M\frac{K^{N-1}(1+K)}{1+K^{N-1}},
					\\
					W_2^2(\mu,\nu) + M^2 \frac{\tond*{1-K^{N-1}}^2(1-K^2)}{\tond*{1-K^{2N-2}}(1-K)^2}\; &\text{if } W_2(\mu,\nu)< M\frac{K^{N-1}(1+K)}{1+K^{N-1}}.
				\end{cases}
                \\
                & \le 2C\frac{1-K^2}{1-K^{2N}} \tond*{K^{N-1}W_2(\mu,\nu)+M\frac{1-K^{N-1}}{1-K}}^2.
            \end{split}
			\end{equation}
	\end{theorem}
	\begin{proof}
		We assume $W_2(\mu,\nu)>0$, otherwise there is nothing to prove.		
				Fix $N\in \N_{\ge 1}$ and for $n =0,\ldots, N-1$ consider any numbers $0 \le \eta_n \le 1$ and such that $\eta_{N-1} =1$.
		Analogously to \cite{alt-che-2023}, we construct now two stochastic processes
		$\tond*{X_n}_{n=0}^N, \tond*{X_n'}_{n=0}^N$ iteratively:
		first, we set $X_0\sim \mu, X_0' \sim \nu$ optimally coupled for $W_2(\mu,\nu)$.
		Then, for $n = 0, \ldots, N-1$, conditionally on $X_n, X_n'$, we define $\tilde X_n = (1-\eta_n)X_n' +\eta_n X_n $ and sample $X_{n+1}, X_{n+1}'$
		optimally coupled for
			$W_2\tond*{\delta_{X_n}P, \delta_{\tilde X_n}P}$, in a measurable way \cite[Corollary 5.22]{vil-2009}.
		In what follows,  we write $\mu'_n =\law(X'_n)$, $\mu_n =\law(X_n) = \mu P^n$ and $\nu_n = \nu P^n$: observe that $\mu'_0 = \nu_0$ and $\mu'_N = \mu_N$ since $\eta_{N-1}=1$.
		For $n = 0,\ldots N-1$
		\begin{equation}
			\begin{split}
				\kldiv{\mu'_{n+1}}{\nu_{n+1}} &\le \kldiv{\mu'_{n}}{\nu_{n}}+\expe{\kldiv{\delta_{\tilde X_n} P}{\delta_{X'_n} P}} \quad \text{ by Theorem 3.1 in \cite{alt-che-2023}}
				\\
				& \leq \kldiv{\mu'_{n}}{\nu_{n}}+ C\expe{{\abs*{\tilde X_n- X'_n}}^2}  \quad \text{ by Assumption \ref{ass:kl-reg-one-step}}
				\\
				& = \kldiv{\mu'_{n}}{\nu_{n}}+ C\eta_n^2\expe{\abs*{ X_n- X'_n}^2}.
			\end{split}
		\end{equation}
		Hence, iterating the above
		\begin{equation}\label{eq:upp-bound-kl-sum-squared-shifts}
			\begin{split}
				\kldiv{\mu_N}{\nu_N}  = \kldiv{\mu'_N}{\nu_N}
				& \le  C \sum_{n=0}^{N-1}\eta_n^2\expe{\abs*{ X_n- X'_n}^2}.
			\end{split}
		\end{equation}

		Observe now that for $n\ge 1$
		\begin{equation}
			\begin{split}
				\expe{\abs*{ X_n - X'_n}^2} & = \expe{W_2^2\tond*{\delta_{X_{n-1}}P, \delta_{\tilde X_{n-1}}P}}
				\\
				& \le \expe{\tond*{K\abs*{X_{n-1}-\tilde X_{n-1}}+M}^2} \quad \text{ by Assumption \ref{ass:weak-curv-wp}}
				\\
				& = \expe{\tond*{K\tond*{1-\eta_{n-1}}\abs*{X_{n-1}- X'_{n-1}}+M}^2}
				\\
				& \le \tond*{K\tond*{1-\eta_{n-1}}\sqrt{\expe{\abs*{X_{n-1}-X'_{n-1}}^2}} +M}^2.
			\end{split}
		\end{equation}
		Taking the square root and iterating we deduce that
				\begin{equation}
			\sqrt{\expe{\abs*{X_n-X'_n}^2} } \leq \tond*{W_2(\mu,\nu)K^n \prod_{k=0}^{n-1}\tond*{1-\eta_k} + M \sum_{k=1}^n K^{n-k}\prod_{j=k}^{n-1}\tond*{1-\eta_j}},
		\end{equation}
		with the convention that empty products are equal to $1$.
		Plugging this back in \eqref{eq:upp-bound-kl-sum-squared-shifts} we find that
		\begin{equation}\label{eq:obj-upp-bound-kl}
			\begin{split}
					\kldiv{\mu_N}{\nu_N} &\le C\sum_{n=0}^{N-1}\eta_n^2 \tond*{W_2(\mu,\nu)K^n \prod_{k=0}^{n-1}\tond*{1-\eta_k} + M \sum_{k=1}^n K^{n-k}\prod_{j=k}^{n-1}\tond*{1-\eta_j}}^2.
			\end{split}
		\end{equation}
		The above holds for any shifts $(\eta_n)_{n=0}^{N-1}$ satisfying $0 \le \eta_n\le 1$ and $\eta_{N-1} =1 $, so we can optimize over them to complete the proof: this is done in Lemma \ref{lem:recursive-problem}, since the minimum of objective functional in  
		\eqref{eq:obj-upp-bound-kl} is equal to $C\cdot S(N,W_2(\mu,\nu))$ with the notation of Lemma \ref{lem:recursive-problem}.
	\end{proof}
	
	\begin{corollary}\label{cor:kl-w-reg-LD}
		For the Langevin semigroup, under Assumption \ref{ass:log-lip-pert-log-conc} and \ref{ass:log-lip-potential} for $T>0$ and $\mu,\nu \in \pp_2\tond*{\R^d}$ we have
		\begin{equation}\label{eq:rev-transp-entr-ld}
			\kldiv{\mu\ldsem_T}{\nu\ldsem_T} \leq \begin{cases}
				& \frac{\tond*{W_2(\mu,\nu) +2LT}^2}{4T} \quad \text{ if } \alpha =0,
				\\
				&\frac{\alpha}{2\tond*{1-e^{-2\alpha T}}}
				\tond*{e^{-\alpha T}W_2(\mu,\nu)+\frac{2L}{\alpha}\bigl(1-e^{-\alpha T}\bigr)}^2 \quad \text{ if } \alpha>0.
			\end{cases}
		\end{equation}
	\end{corollary}
	\begin{proof}
		We work with the discrete time approximation \eqref{eq:LMC}, with a step-size $h =\frac{T}{N}$ for $N\geq \beta T$.
		Observe that 
		\begin{align*}
			\tond*{\lmcsem{h}}^N = \tond*{\gdsem{h}\gausem{2h}}^N = \gdsem{h} \tond*{\gausem{2h}\gdsem{h}}^{N-1} \gausem{2h}
		\end{align*}
		and by the data processing inequality 
				\begin{align*}
			\kldiv{\mu \tond*{\lmcsem{h}}^N}{ \nu \tond*{\lmcsem{h}}^N} & \le  \kldiv{ \mu\gdsem{h} \tond*{\gausem{2h}\gdsem{h}}^{N-1} }{\nu \gdsem{h} \tond*{\gausem{2h}\gdsem{h}}^{N-1}}.
		\end{align*}
		The one step transition kernel 
		\[
				 \tond*{\gausem{2h}\gdsem{h}}
		\]
		now satisfies Assumption \ref{ass:kl-reg-one-step} with $C=\frac{1}{4h}$ since
        \[
            \kldiv{\delta_x \gausem{2h}\gdsem{h}}{\delta_y \gausem{2h}\gdsem{h}} \leq \kldiv{\delta_x \gausem{2h}}{\delta_y \gausem{2h}} \le \frac{\abs*{x-y}^2}{4h}.
        \]
        Furthermore, Assumption \ref{ass:weak-curv-wp} holds with $p=2$,
        $K = 1-\alpha h$ and $M = 2Lh$, as in the proof of Lemma \ref{lem:LMC-PS-weak-curvature}. The conclusion follows then from Theorem \ref{thm:reverse-transport-entropy-discrete-time}, applied to the initial measures $(\mu\gdsem{h},\nu\gdsem{h})$ with $N-1$ steps, and finally letting $N\to \infty$, recalling the lower semicontinuity of the relative entropy \cite[Lemma 9.4.3]{amb-gig-sav-2008} and that
		$\mu \tond*{\lmcsem{h}}^N \to \mu\ldsem_T$ weakly as $N \to \infty$.

	\end{proof}
	
	\begin{remark}
		Reverse transport-entropy inequalities for the Langevin dynamics \eqref{eq:LD} beyond strong log-concavity were also studied in the very recent preprint \cite{lu-wan-2025}. The main Assumption of \cite{lu-wan-2025}  is that for constants $R,M,m>0$
        \[
            -\scal*{\nabla U(x)-\nabla U(y), x-y}\le\begin{cases} -m\abs*{x-y}^2 
            \text{ if } \abs*{x-y}\ge R,
            \\
            M\abs*{x-y}^2 
            \text{ if } \abs*{x-y}\le R.
            \end{cases}
        \]
        This is similar in spirit to Assumption \ref{ass:log-lip-pert-log-conc}, by requiring a contraction whenever $\abs*{x-y}$ is large enough, but the assumptions are not directly comparable. One advantage of the bound in \cite{lu-wan-2025} is that, unlike \eqref{eq:rev-transp-entr-ld}, $\kldiv{\mu \ldsem_t}{\nu\ldsem_t}$ decays to $0$ (exponentially fast) as $t\to \infty$, reflecting the ergodicity of the process.
        However, the constants involved exhibit a less favorable behavior with respect to the problem parameters $R,m,M$, and
         Corollary \ref{cor:kl-w-reg-LD} is thus more useful for our applications to cutoff, where we only need to exploit a finite-time regularization.
	\end{remark}
	We also have an analogous result for the Proximal Sampler.
	 Here, we only record a looser bound for readability, and since it suffices for our purposes and most applications.
	\begin{corollary}\label{cor:kl-w-reg-PS}
		For the Proximal Sampler, under Assumption \ref{ass:log-lip-pert-log-conc} for $N\in \N_{\ge 1}$ and $\mu,\nu \in \pp_2\tond*{\R^d}$ we have
		\begin{equation}
			\kldiv{\mu\tond*{\pssem{h}}^N}{\nu\tond*{\pssem{h}}^N} \leq 
			\begin{cases}
				\frac{\tond*{W_2(\mu,\nu)+2LNh}^2}{2Nh} \quad \text{if } \alpha = 0
				\\
				\frac{\alpha(2+\alpha h)}{(1+\alpha h)^{2N}-1}\tond*{W_2(\mu,\nu)+\frac{2L}{\alpha}\tond*{(1+\alpha h)^{N-1}-1}}^2\quad \text{if } \alpha > 0.
			\end{cases}
		\end{equation}
	\end{corollary}
	\begin{proof}
		We have for $x,y\in \R^d$ and $h>0$
		\begin{align*}
			\kldiv{\delta_x \pssem{h}}{\delta_y \pssem{h}} & = \kldiv{\delta_x \gausem{h}\revgausem{h}}{\delta_y \gausem{h}\revgausem{h}}
			\\
			& \le \kldiv{\delta_x \gausem{h}}{\delta_y \gausem{h}}
			\\
			& = \frac{\abs*{x-y}^2}{2h},
		\end{align*}
		so Assumption \ref{ass:kl-reg-one-step} is satisfied with $C = \frac{1}{2h}$. Recalling Lemma \ref{lem:LMC-PS-weak-curvature},
		Assumption \ref{ass:weak-curv-wp} holds for $p=2$, $K=\frac{1}{\alpha h+1}$ and $M=\frac{2Lh}{1+\alpha h}$, and
        the conclusion follows easily from Theorem \ref{thm:reverse-transport-entropy-discrete-time} and elementary simplifications.
	\end{proof}

	\section{Applications to cutoff}
    \label{sec:cutoff-criteria}

    Following the systematic approach developed in the series of works \cite{sal-2023,sal-2024,sal-2025-diffusions,ped-sal-2025}, and in particular the recent paper \cite{ped-sal-2025-bis},
    our goal is to estimate the width of the mixing window $\wmix(\eps) = \tmixa{\eps}-\tmixa{1-\eps}$ for any fixed $\eps \in (0,\frac12)$, leveraging the results of the previous sections.
    To this end, we first need the following lemma (a straightforward adaptation of Theorem 1 in \cite{ped-sal-2025-bis}): combined with the results of Section \ref{sec:def-loc-conc}, it allows us to upgrade the known upper bound on the total variation distance $\totvar{\mu_{t_0}}{\pi} \le 1-\eps$ at the mixing time $t_0 = \tmixa{1-\eps}$ to an upper bound on the Wasserstein distance $W_p(\mu_{t_0},\pi)$.	
	\begin{lemma}[W-TV transport inequality]\label{lem:WTV}
		Let $\mu_1,\mu_2\in\pp\tond*{\Omega}$, $p\ge1$ and $C_1,M_1,C_2,M_2\ge 0$ such that 
		$\mu_i \in \tpha{C_i,M_i}$.
		Then, if $\totvar{\mu_1}{\mu_2}<1$ we have that
		\begin{equation*}
			W_p(\mu_1,\mu_2)  \le \tond*{\sqrt{C_1}+\sqrt{C_2}}\sqrt{2\log\frac{1}{1-\totvar{\mu_1}{\mu_2}}}+M_1+M_2.
		\end{equation*}
	\end{lemma}
    The rest of the argument consists in applying a reverse transport-entropy inequality as in Theorem \ref{thm:reverse-transport-entropy-discrete-time} (when $p=2$), in order to further upgrade the upper bound on the Wasserstein distance into an upper bound on the relative entropy $\kldiv{\mu_{t_0+s}}{\pi}$, at the price of waiting an additional time $s >0$.
    Finally, an upper bound on $\tmixa{\eps}$ is readily obtained, recalling that mixing in total variation distance occurs quickly once the relative entropy is small (cf. Lemma 7 in \cite{sal-2023}), in terms of the spectral gap of $P$.
    
    We illustrate now this approach more concretely on our two model examples, the Langevin dynamics and the Proximal Sampler.
    Before proceeding, we recall from \cite{bri-ped-2024} the following bound on the Poincaré constant under Assumption \ref{ass:log-lip-pert-log-conc}, which allows us to state also a bound depending only on $\alpha$ and $L$.
    \begin{lemma} \label{lem:cpi-pi-log-lip-pert}
        Under Assumption \ref{ass:log-lip-pert-log-conc} for $\alpha>0$, we have 
        \[
            \cpi{\pi} \le \frac{1}{\alpha} \expo{\frac{L^2}{\alpha}+\frac{4L}{\sqrt{\alpha}}}.
        \]
    \end{lemma}

	\subsection{Langevin dynamics}
    
    Our main result in this section for the Langevin dynamics is the following estimate on the total variation mixing window.
    \begin{theorem} \label{thm:cutoff-LD-general}
        Under Assumption \ref{ass:log-lip-pert-log-conc} and \ref{ass:log-lip-potential} for $\alpha >0$, consider a probability measure $\mu_0 \in \defttha{J,S}$ (for $J,S\ge 0$) as initialization for \eqref{eq:LD}, and fix $\eps\in (0,\frac12)$. Then, 
        \begin{equation}
        \begin{split}
            \wmixa{\mu_0,\eps} &\lesiep \cpi{\pi}\tond*{\frac{L^2}{\alpha}+1}+\sqrt{\cpi{\pi}}\tond*{e^{-\alpha \tmixa{\mu_0,1-\eps}}(\sqrt{J}+S)+\frac1{\sqrt{\alpha}}+\frac{L}{\alpha}}
            \\
            & \lesiep \expo{\frac{L^2}{\alpha}+\frac{4L}{\sqrt{\alpha}}}\cdot \frac1\alpha\tond*{1+\frac{L^2}{\alpha}+\frac{L}{\sqrt{\alpha}}+{e^{-\alpha \tmixa{\mu_0,1-\eps}}(\sqrt{J}+S)}{\sqrt{\alpha}}}.
        \end{split}
        \end{equation}
    \end{theorem}

    \begin{proof}
        Fix $\eps \in (0,\frac12)$ and set $t_0 \coloneqq \tmix\tond*{\mu_0,1-\eps}$.
        By Corollary \ref{cor:local-defective-transport-entropy-LD}, 
        \begin{align}
            \mu_{t_0} &\in \defttha{e^{-2\alpha t_0}J +\frac1\alpha, e^{-\alpha t_0}S+\frac{2L}{\alpha}},
            \\
            \pi &\in \defttha{\frac1\alpha, \frac{2L}{\alpha}}.
        \end{align}
        Hence, by Lemma \ref{lem:WTV}
        \[
            W_2\tond*{\mu_{t_0},\pi} \lesiep e^{-\alpha t_0}(\sqrt{J}+S)+\sqrt{\frac{1}{\alpha}}+\frac{L}{\alpha}.
        \]
         By Corollary \ref{cor:kl-w-reg-LD} we deduce that for any $s> 0$
        \[
            \kldiv{\mu_{t_0+s}}{\pi} \lesiep 
            \frac{\alpha}{e^{2\alpha s}-1}\tond*{e^{-2\alpha t_0}(J+S^2)+{\frac{1}{\alpha}}+\frac{L^2}{\alpha^2}}+\frac{L^2}{\alpha}.
        \]
        Hence, by \cite[Lemma 2]{sal-2025-diffusions} for any $s> 0$ we have that
        \begin{align} \label{eq:ld-wmix-to-optimize-bound}
                	\wmixa{\mu_0,\eps} & \lesiep  s+
                \cpi{\pi}\quadr*{1+
                \frac{\alpha}{e^{2\alpha s}-1}\tond*{e^{-2\alpha t_0}(J+S^2)+{\frac{1}{\alpha}}+\frac{L^2}{\alpha^2}} +\frac{L^2}{\alpha}}
                \\
                \label{eq:ld-wmix-to-optimize-bound-bis}
                & \lesiep s+\cpi{\pi}\tond*{1+\frac{L^2}{\alpha} }+\frac{\cpi{\pi}}{s}\tond*{e^{-2\alpha t_0}(J+S^2)+{\frac{1}{\alpha}}+\frac{L^2}{\alpha^2}},
        \end{align}
        where we have used that $\frac{\alpha}{e^{2\alpha s}-1}\le \frac{1}{2s}$ for $s>0$ in the last line.
        We can now make the choice $s = \sqrt{\cpi{\pi}\tond*{e^{-2\alpha t_0}(J+S^2)+{\frac{1}{\alpha}}+\frac{L^2}{\alpha^2}}}$, which yields
        \[
                	\wmixa{\mu_0,\eps} \lesiep \cpi{\pi}\tond*{\frac{L^2}{\alpha}+1}+\sqrt{\cpi{\pi}}\tond*{e^{-\alpha t_0}(\sqrt{J}+S)+\frac1{\sqrt{\alpha}}+\frac{L}{\alpha}},
        \]
        This is the first bound in the theorem, and the second one follows by substituting the estimate for $\cpi{\pi}$ from Lemma \ref{lem:cpi-pi-log-lip-pert}.
    \end{proof}
    \begin{remark}
        A better estimate, at the expense of readability, can be obtained by optimizing over $s\ge 0$ in \eqref{eq:ld-wmix-to-optimize-bound}.
    \end{remark}

    To appreciate the above result, we compare it with the recent \cite[Theorem 2]{sal-2025-diffusions}, which states that if $\pi$ is $\alpha$-log-concave (for $\alpha>0$) and the initialization $\mu_0 =\delta_x$ is deterministic, then $\wmixa{\eps} \lesiep \frac1\alpha$: hence, if $\alpha>0$ is bounded from below, the mixing window is remarkably of constant order. In contrast, in many situations it is easily seen that the mixing time 
    diverges to $\infty$ as the dimension $d\to \infty$, and this leads therefore to a cutoff phenomenon. A classical toy example is given by the Ornstein--Uhlenbeck process in $\R^d$ initialized at $\mu_0 = \delta_0$, for which  \cite[Theorem 2]{sal-2025-diffusions} correctly predicts $\wmixa{\eps}\lesiep 1$, while it is easily seen that $\tmixa{\eps}\to \infty$ as $d \to \infty$ for fixed $0<\eps<1$.
    Theorem \ref{thm:cutoff-LD-general} above generalizes this phenomenon to the class of log-Lipschitz perturbations of strongly log-concave measures, as in Assumption \ref{ass:log-lip-pert-log-conc}, by proving that $\wmixa{\eps} \lesiep C(L,\alpha)$, i.e. that for any fixed $\eps\in (0,\frac12)$ the mixing window is uniformly bounded by a positive constant that depends only on $L$ and $\alpha$. 
    For concreteness, we consider below a simple example in this class.
\begin{example}
    Consider in $\R^d$ a target
    $ \pi_d=\nu_d*\mathcal{N}\tond*{0,I_d}$
    for a distribution $\nu_d\in\pp\tond*{\R^d}$ whose support is
    contained in the centered Euclidean ball $B\tond*{0,R}$ of radius $R>0$,
    for some fixed $R>0$ independent of the dimension.
    Observe that $\pi_d$ can be highly non-log-concave when $R>1$, as
    in the simple case where
    $  \nu_d=\frac12\tond*{\delta_{-Re_1}+\delta_{Re_1}}$,
    with $\tond*{e_i}_i$ the standard basis of $\R^d$.
    However, following \cite[Example 2.3]{bri-ped-2024} (see the
    arXiv version), we can write $\pi_d\propto e^{-U_d}$ with
    \[
        U_d(x)=\frac12\abs*{x}^2+H_d(x),
        \qquad 
        \abs*{\nabla H_d(x)}\le R.
    \]
    In particular, Assumptions \ref{ass:log-lip-pert-log-conc} and \ref{ass:log-lip-potential} hold
    with $\alpha=\beta=1$ and $L=R$.
    An application of Theorem \ref{thm:cutoff-LD-general} implies then
    that, from any deterministic initialization,
    $ \wmixa{\eps}\lesiep C(R)$ for a constant $C(R)>0$ depending only on $R$, and not on the
    dimension $d$. In contrast, for the deterministic initialization
    $\delta_0$ at the origin, for example, it is easy to see that the mixing time
    goes to infinity.
    Indeed, fix a finite time horizon $T>0$. Using a synchronous
    coupling between the Langevin dynamics and the
    Ornstein--Uhlenbeck process, and using that the drift $\nabla H_d$
    is bounded, it is easy to see that
    $
        W_\infty\tond*{
            \delta_0P_T,
            \mathcal{N}\tond*{0,\tond*{1-e^{-2T}}I_d}
        }
        \le R
    $
    (cf. also the proof of \cite[Proposition 1.1]{khu-maa-ped-2024}).
    Recalling the definition of $\pi_d$ as a convolution, this holds for $T= \infty$ too, i.e. $
        W_\infty\tond*{\pi_d, \mathcal{N}\tond*{0,I_d}}\le R$.
    Now choose any $c\in\tond*{\sqrt{1-e^{-2T}},1}$. Since the Euclidean norm of a standard Gaussian vector $Z_d\sim \mathcal{N}(0,I_d)$ concentrates around $\sqrt{d}$, i.e. $\frac{\abs*{Z_d}}{\sqrt{d}} \to 1$ in probability as $d\to \infty$, 
    we deduce from the two
    $W_\infty$ bounds above that
    \[
        \delta_0 P_T\tond*{B\tond*{0,c\sqrt d}}\to 1,
        \qquad \pi_d\tond*{B\tond*{0,c\sqrt d}} \to 0
    \]
    as $d\to\infty$. In particular, it follows that, for every fixed $T>0$,
    $\totvar{\delta_0 P_T}{\pi_d} \to 1$ as $d\to \infty$, as desired.

    To sum up, we have therefore established the occurrence of a cutoff
    phenomenon beyond the log-concave setting: as
    $d\to\infty$, the mixing time goes to infinity, whereas the mixing
    window is uniformly bounded by a constant which depends only on $R$ and $\eps$.
\end{example}

	\subsection{Proximal sampler}
    In this section, we derive an estimate for the mixing window for the Proximal Sampler, analogous to the one for the Langevin dynamics.
	Here, for simplicity and to lighten the notation, we restrict ourselves to the case where the initialization $\mu_0 \in \defttha{\frac1\alpha, \frac{6L}{\alpha}}$ so that 
	Corollary \ref{cor:local-defective-th-proximal-sampler} ensures that this remains true for $\mu_0 \tond*{\pssem{h}}^N$ for every $N\ge 0$.
	Also, following \cite{ped-sal-2025-bis},  we introduce the quantity
	\begin{eqnarray*}
		\hatcpi{\pi}  \coloneqq  1+\frac{\cpi{\pi}}{h},
	\end{eqnarray*}
	which plays the same role as $\cpi{\pi}$ did for the Langevin dynamics.

	\begin{theorem}
		Under Assumption \ref{ass:log-lip-pert-log-conc} for $\alpha>0$, consider a probability measure $\mu_0 \in \defttha{\frac1\alpha, \frac{6L}{\alpha}}$  as initialization for the Proximal Sampler, and fix $\eps\in (0,\frac12)$. Then, 
		\begin{equation}
			\wmixa{\mu_0,\eps}  \lesiep \hatcpi{\pi}\tond*{1+\frac{L^2(1+\alpha h)}{\alpha} }+\sqrt{\frac{\hatcpi{\pi}(1+\alpha h)}{h}\tond*{\frac{1}{\alpha}+\frac{L^2}{\alpha^2}}}.
		\end{equation}
	\end{theorem}
	
	\begin{proof}
	Write  $k_0 = \tmix\tond*{\mu_0, 1-\eps}$.
	Since $\mu_{k_0}, \pi \in \defttha{\frac1\alpha, \frac{6L}{\alpha}}$, 
	 by Lemma \ref{lem:WTV} we deduce that 
	\[
	W_2\tond*{\mu_{k_0},\pi} \lesiep \sqrt{\frac{1}{\alpha}}+\frac{L}{\alpha}.
	\]
	By Corollary \ref{cor:kl-w-reg-PS} we deduce that for any $n\in \N_{\ge 1}$
	\[
	\kldiv{\mu_{k_0+n}}{\pi} \lesiep 
	\frac{1+\alpha h}{nh} \tond*{{\frac{1}{\alpha}}+\frac{L^2}{\alpha^2}}+\frac{L^2(1+\alpha h)}{\alpha}
	\]
    where we have used that $(1+\alpha h)^{2n}\ge 1+ 2n\alpha h$.
	Recalling \cite[Lemma 2.5]{ped-sal-2025-bis}, it follows that
	\begin{align} \label{eq:ld-wmix-to-optimize-bound-ps}
			\wmixa{\mu_0,\eps} & \lesiep  n+
		\hatcpi{\pi}\quadr*{1+\frac{1+\alpha h}{nh} \tond*{{\frac{1}{\alpha}}+\frac{L^2}{\alpha^2}}+\frac{L^2(1+\alpha h)}{\alpha}}
		\\
		\label{eq:ld-wmix-to-optimize-bound-ps-bis}
		& \lesiep n+\hatcpi{\pi}\tond*{1+\frac{L^2(1+\alpha h)}{\alpha} }+\frac{\hatcpi{\pi}(1+\alpha h)}{nh}\tond*{\frac{1}{\alpha}+\frac{L^2}{\alpha^2}}.
	\end{align}
	Choosing $n = \ceil*{\sqrt{\frac{\hatcpi{\pi}(1+\alpha h)}{h}\tond*{\frac{1}{\alpha}+\frac{L^2}{\alpha^2}}}}$ yields
    \[
       	\wmixa{\mu_0,\eps} \lesiep \hatcpi{\pi}\tond*{1+\frac{L^2(1+\alpha h)}{\alpha} }+\sqrt{\frac{\hatcpi{\pi}(1+\alpha h)}{h}\tond*{\frac{1}{\alpha}+\frac{L^2}{\alpha^2}}}.
    \]
	\end{proof}

	\subsection*{Acknowledgments}
	The author thanks Sam Power, Justin Salez, and Peter Whalley for useful comments and references.

	\printbibliography
    \appendix

    \newpage
	
	\section{Technical details }

    \begin{lemma}\label{lem:recursive-problem}
		For $N \in \N_{\ge 1}$, $0\le K\le 1$ and $M,A\ge 0$, define
		\begin{equation}
			S(N,A) = \min_{\substack{0\le \eta_{n}\le 1 \\ \eta_{N-1} = 1}} \sum_{n=0}^{N-1} \eta_n^2 A_n^2
		\end{equation}
		where 
		\begin{align}
			A_0 &= A
			\\
			A_{n+1} & = K (1-\eta_n) A_n +M.
		\end{align}
		\begin{enumerate}
			\item If $K=1$, we have for $N\ge 1$ that
			\begin{equation}
				\begin{split}
					S(N,A) &= \begin{cases}
						\frac{1}{N}\tond*{A+(N-1)M}^2 \quad &\text{ if } A\ge M
						\\
						A^2 + (N-1)M^2 \quad &\text{ if } A< M
					\end{cases}
					\\
					& \leq \frac{2}{N}\tond*{A+(N-1)M}^2.
				\end{split}
			\end{equation}
			\item If $0\le K<1$, we have that $S(1,A)=A^2$ and for $N\ge 2$ 
			\begin{equation}
				\begin{split}
					S(N,A) &= \begin{cases}
						\frac{1-K^2}{1-K^{2N}} \tond*{K^{N-1}A+M\frac{1-K^{N-1}}{1-K}}^2 \quad &\text{ if } A\ge M\frac{K^{N-1}(1+K)}{1+K^{N-1}}
						\\
						A^2 + M^2 \frac{\tond*{1-K^{N-1}}^2(1-K^2)}{\tond*{1-K^{2N-2}}(1-K)^2}\quad &\text{ if } A< M\frac{K^{N-1}(1+K)}{1+K^{N-1}}
					\end{cases}
					\\
					& \leq 2\frac{1-K^2}{1-K^{2N}} \tond*{K^{N-1}A+M\frac{1-K^{N-1}}{1-K}}^2.
				\end{split}
			\end{equation}
		\end{enumerate}
	\end{lemma}
	\begin{proof}
		The proof is by induction; the case $N=1$ is trivial.
		If $N\ge 2$, using 
		the definition of $A_n$ we see that 
		\begin{equation}\label{eq:recursive-step-optimization}
			\begin{split}
				S(N,A) &= \min_{0\le \eta_0 \le 1} \eta_0^2 A^2 + S(N-1, K(1-\eta_0) A+M).
			\end{split}
		\end{equation}
		We now distinguish between different cases, and assume that $A>0$
        (the case $A=0$ is easy noticing that $S(N,0) =  S (N-1,M)$).
		\begin{enumerate}
			\item 	\textbf{Case 1:} $K=1$. 
			
			Notice that for any $0\le \eta_0 \le 1$ we have $(1-\eta_0) A+M \ge M$. Hence, using the inductive hypothesis 
			\[
			S(N,A) = \min_{0\le \eta_0 \le 1} \eta_0^2 A^2 + \frac{1}{N-1}\tond*{(1-\eta_0)A+(N-1)M}^2.
			\]
			If $A\ge M$, the above is minimized for the value
			\[
			\eta_ 0 \coloneqq \frac{A+(N-1)M}{NA} \in [0,1],
			\]
			and the first bound follows.
			Otherwise, the minimum is achieved for $\eta_0 =1$, and the other bound easily follows.
			 To prove the final inequality valid in both regimes, it suffices to show that 
            \[
                f(A) \coloneqq \frac{2}{N}\tond*{A+(N-1)M}^2 - A^2 - (N-1)M^2  \ge 0 
            \]
            if $0 \le A\le M$. This is easily done, since for $N\ge 2$ the function $f$ is quadratic with non-positive coefficient for $A^2$, and so it suffices to check the claim for $A=0$ and $A =M$ (and separately for $N=1$).
			
			\item \textbf{Case 2:} $0\le K<1$.
			
			Observe first that for any $0 \le \eta_0\le 1$ we have that 
			\[
			K(1-\eta_0)A + M \ge M \ge M\frac{K^{N-2}(1+K)}{1+K^{N-2}}
			\]
            since $0\le K<1$.
			Hence, by the inductive hypothesis,
			\[
			\begin{split}
				S(N,A) & = \min_{0\le \eta_0 \le 1} \eta_0^2 A^2 + S(N-1, K(1-\eta_0) A+M).
				\\
				& = \min_{0\le \eta_0 \le 1} \eta_0^2 A^2 +\frac{1-K^2}{1-K^{2(N-1)}} \tond*{K^{N-2}\tond*{K(1-\eta_0)A+M}+M\frac{1-K^{N-2}}{1-K}}^2
				\\
				& = \min_{0\le \eta_0 \le 1} \eta_0^2 A^2 +\frac{1-K^2}{1-K^{2(N-1)}} \tond*{K^{N-1}(1-\eta_0)A+M\frac{1-K^{N-1}}{1-K}}^2.
			\end{split}
			\]
			The minimum for the corresponding unconstrained problem is achieved for 
			\[
			\eta_0 \coloneqq \frac{(1-K^2)K^{N-1}}{A\tond*{1-K^{2N}}}\tond*{K^{N-1}A+M \frac{1-K^{N-1}}{1-K}}.
			\]
			This is admissible (i.e. lies in $[0,1]$) if and only if
			\[
			A\ge M\frac{K^{N-1}(1+K)}{1+K^{N-1}},
			\]
			in which case we establish the bound for the first regime after substituting this value for $\eta_0$. 
			In the opposite case, the minimum is achieved for $\eta_0 = 1$, and we obtain the bound in the second regime.
            To prove the final inequality, it suffices to show that 
            \[
                f(A) \coloneqq 2\frac{1-K^2}{1-K^{2N}} \tond*{K^{N-1}A+M\frac{1-K^{N-1}}{1-K}}^2 - A^2 - M^2 \frac{\tond*{1-K^{N-1}}^2(1-K^2)}{\tond*{1-K^{2N-2}}(1-K)^2} \ge 0 
            \]
            for $0 \le A\le M\frac{K^{N-1}(1+K)}{1+K^{N-1}}$. This is easily done, since for $N\ge 2$ the function $f$ is quadratic with non-positive coefficient for $A^2$, and so it suffices to check the claim for $A=0$ and $A =M\frac{K^{N-1}(1+K)}{1+K^{N-1}}$ (and separately for $N=1$).
		\end{enumerate}

	\end{proof}

\end{document}